\documentclass{article}
\usepackage{amssymb}
\usepackage{amsmath}
\usepackage{amsthm}
\usepackage{verbatim}
\usepackage{color}
\usepackage{url}
\usepackage{fancyhdr}

\newtheorem{thm}{Theorem}
\newtheorem{prop}{Proposition}
\newtheorem{lem}{Lemma}
\newtheorem{cor}{Corollary}

\newtheorem{defn}{Definition}

\newtheorem{rem}{Remark}
\newtheorem{exmp}{Example}

\newcommand{\sparen}[1]{\left(#1\right)}

\numberwithin{equation}{section}

\begin{document}
\title{Unique Determination of Sound Speeds for Coupled Systems of Semi-linear Wave Equations\\\vskip 0.8cm}
\author{Alden Waters \thanks{JBI Institute, Rijksuniversiteit Groningen, Groningen, Netherlands} \hspace{0.05cm}}

\date{}

\maketitle \vskip 0.5cm
\begin{abstract}
We consider coupled systems of semi-linear wave equations with different sound speeds on a finite time interval $[0,T]$ and a bounded Lipschitz domain $\Omega$ in $\mathbb{R}^3$ with boundary $\partial\Omega$. We show the coupled systems are well posed for variable coefficient sounds speeds and short times. Under the assumption of small initial data, we prove the source to solution map associated with the nonlinear problem is sufficient to determine the source to solution map for the linear problem. We can then reconstruct the sound speeds in $\Omega$ for the coupled nonlinear wave equations under certain geometric assumptions. In the case of the full source to solution map in $\Omega\times[0,T]$ this reconstruction could also be accomplished under fewer geometric assumptions. 
\end{abstract}



\section{Introduction:}
We consider coupled systems of semi-linear wave equations with variable sound speeds on a bounded Lipschitz domain $\Omega$ in $\mathbb{R}^3$ with boundary $\partial\Omega$. In nonlinear problems, when waves are propagated, they interact and the interaction may cause difficulties in building an accurate parametrix and detecting the variable coefficients.

For the problem of elasticity, the stress the material is under going is described by the Lam\'{e} parameters, $\lambda$ and $\mu$. Recently in \cite{PUV} it was shown that this important linear hyperbolic problem where the solutions are vector valued can be reduced to three variable speed wave equations with scalar valued solutions. The authors of \cite{PUV} are then able to solve the associated inverse boundary value problem for the linear elasticity equation by building solutions to the wave equations. Ultimately we hope to consider the fully nonlinear elastic wave equations, which they do not consider in \cite{PUV}, but we will report on this in future work. However, even in the simpler model here, for the case of variable sound speeds well posedness estimates appear to be novel. Parametric construction of solutions to these coupled systems has been done only for the constant coefficient case c.f., \cite{katayama,katayama2, sideris3, sideris2}. In the case of nonlinear elasticity, constant coefficient equations have been examined in \cite{sideris,met1,met2,met3} although many of these references are interested in a different (and challenging!) perspective which is the issue of well-posedness and scattering for long times.   

The problem of parameter recovery is well studied for a class of linear hyperbolic problems such as the wave equation $(\partial_t^2-\Delta_g)u=0$, for generic Riemannian manifolds $(M_0,g)$ c.f. \cite{DSK,SU,SK,eskin2,BD,CS} for example. One can even recover the metric $g$ for the associated semi-linear problem. The latter problem is handled via a linearisation method, \cite{LUK2}. The authors also apply their linearisation techniques to the case of Einstein's equations in the related article \cite{LUK}. The difference in these articles and the material presented here is that the coefficients e.g the metric $g$ are time dependent, and ours are not. Time dependence of the metric $g$ adds considerable difficulties. However we are able to handle the case of multiple sound speeds and coupled systems of nonlinear wave equations. Due to the technical difficulties of the problem, such coupled nonlinear wave equations have not been considered before.

Even in the linear case, the pioneering work on parameter recovery in nonlinear inverse problems in \cite{LUK2,LUK} uses the singularities of their nonlinear hyperbolic problems to determine the metric in their partial differential equations (PDE). They use the calculus of cononormal singularities developed in \cite{MR} and \cite{MU} to recover the metric at every point. With $(M_0,g)$ a Lorentzian manifold, each of these articles \cite{LUK2,LUK} makes use of the Lagrangian distributions in \cite{MU} which are associated to the solutions 
\begin{align}
(\partial_t^2-\Delta_g)u=\delta_{x=x_0}
\end{align}
and then builds general solutions to the wave equation with source terms $f(t,x)$ 
\begin{align}
(\partial_t^2-\Delta_g)u=f(t,x)
\end{align}
by using a Green's function argument. We show that given sufficient regularity in the source data $f(t,x)$ that this approach is unnecessary. Indeed, in \cite{LUK2, LUK}, they assume that the regularity of the source is $f(t,x)\in H^6(V)$, in the open set $V$ in space-time where they are measuring. The regularity they require is actually higher than the regularity needed here, but for their main argument it seems that a distributional solution would suffice. Hence one of the open challenges is to determine how little regularity is needed for metric recovery in the various cases presented here and in \cite{LUK2,LUK}. 

In, \cite{LUK,LUK2} it is necessary to check the interaction of the singularities under the nonlinearity as we know by \cite{rr} that waves can interact when a nonlinearity is present and produce more singularities. Moreover in \cite{rr}, \cite{rr2}, they showed that these crossings are the only place where new singularities can form.  

The second reason one must check the interaction of the nonlinear waves is that certain nonlinearities known as null forms, can act on propagating waves by smoothing them in such a way that approximate solutions to the nonlinear PDE and the linear PDE are indistinguishable in the micro-local sense, c.f., \cite{klainerman1, klainerman2}. In their articles \cite{LUK, LUK2}, the authors exploit the singularity crossings to reconstruct the geometry of domains they consider. We do not, so while we still have to check the interaction of the crossings, we can proceed differently than in  \cite{LUK,LUK2}. As in \cite{LUK,LUK2} they have chosen sufficiently regular data for the PDE, we render this approach is unnecessary, in the time independent coefficient case. 

We have to be careful about the type of measurements that we are taking. In particular, it is not known if the coupled nonlinear equations are well posed for generic compact manifolds with boundary. In fact for quadratic nonlinearities, it is likely that they are not, as the simpler case of the scalar semi-linear wave equation in not globally well posed. We could extend our short time well-posedness estimates to generic globally hyperbolic manifolds, but we leave this for future work. 

As such, the major contributions of this article are the following:
\begin{itemize}
\item A reduction of source-to-solution map data (to co-dimension 1) required to determine the topological structure of the manifold. 
\item Simplification of the singularity analysis and parametrix construction for semi-linear wave equations. 
\item Provision of a toy model and well-posedness estimates for the non-linear elasticity equations. 
\end{itemize}

In order to avoid difficulties with boundary considerations we examine the solutions on the boundary of $[0,T]\times \Omega$, where $T$ is finite. This scenario is not a traditional boundary value problem. The hyper surface $\partial\Omega$ is not a true boundary for the waves, simply where we are measuring. 

Under these same geometric assumptions as in \cite{PUV}, for the nonlinear case, and a small displacement field, we are able to reduce the amount of data required to uniquely determine the vector field to just boundary valued data on the artificial surface $[0,T]\times\partial\Omega$. This result is completely new for nonlinear hyperbolic PDE, even in the case when the solutions are scalar valued. The techniques required for the reduction of data, are new from those in \cite{LUK, LUK2}. 

\noindent \emph{Notation:} 
In this paper we use the Einstein summation convention. For two matrices $A$ and $B$, the inner product is denoted by
 \begin{align*}
A:B=a_{ij}b_{ji},
\end{align*}
and we write $|A|^2 = A:A$. Again, $\overline{\Omega}\subset \mathbb{R}^3$ is a compact subset of $\mathbb{R}^3$. 
For vector--valued functions
\begin{align*}
f(x)=(f_1(x),f_2(x),f_3(x)):\overline{\Omega}\rightarrow \mathbb{R}^3\ ,
\end{align*}
the Hilbert space $H^{m}(\overline{\Omega})^3$, $m\in \mathbb{N}$ is defined as the completion of the space $\mathcal{C}_c^{\infty}(\overline{\Omega})^3$ with respect to the norm
\begin{align*}
\|f\|^2_m = \|f\|^2_{m,\overline{\Omega}} = \sum\limits_{|i|=1}^m \int\limits_{\overline{\Omega}}\sparen{|\nabla^i f(x)|^2+|f(x)|^2}\,dx,
\end{align*}
where we write $\nabla^i= \partial^{i_1} \partial^{i_2} \partial^{i_3}$ for $i=(i_1,i_2, i_3)$ for the higher-order derivative. 

In general, we assume the sound speed coefficients are $C^s(\overline{\Omega})$ with $s$ an integer such that $s-1>3/2$ in order to use Sobolev embedding. We consider the $3-d$ case here, but many of the results generalise to other dimensions and different types of power semi-linearities provided the underlying equations are well-posed. Let $m_1$ and $m_0$ be nonzero constants with $m_1\geq m_0$. We define the admissible class of conformal factors depending on $s$ as
 \begin{align}\label{below}
\mathcal{A}_0^s=\{ c^2(x); \quad m_1\geq c^2(x) \geq m_0;\,\, \forall x\in \overline{\Omega} \quad \mathrm{and}\quad c^2\in C^s(\overline{\Omega})\}
\end{align}
We consider a coupled system with three sound speeds $c_i^2$. We assume $c_i^2\in \mathcal{A}^s_0$ for all $i=1,2,3$. Moreover we also assume there exists a ball $\overline{\Omega}\subset B_R(0)$ such that $c_i\equiv 1$ on $(B_R(0))^c$, and that $c_i$ is extended in a smooth way outside $\overline{\Omega}$ so this is possible. We let $\Omega'$ be an extended domain containing $\overline{\Omega}$.  

\section{Statement of the Main Theorem}

We now examine a coupled system of semi-linear wave equations, which is a toy model from the linearisation of the nonlinear elasticity problem. We could extend these results with appropriate modifications to arbitrary quadratic nonlinearities. Recall we have the following inclusions $\Omega\subset \Omega'\subset\mathbb{R}^3$. Let $u=(u_1,u_2,u_3)$ and we consider the system:
\begin{align}\label{m2N}
&\partial_t^2u_i-c_i^2(x)\Delta u_i=|u|^2+f(t,x)  \quad \textrm{in}\,\, (0,T)\times \Omega', \quad i=1,2,3 \\& \nonumber
u(0,x)=b_0(x) \quad \partial_tu(0,x)=b_1(x) \quad \textrm{in}\quad \Omega' \\&
u(t,x)|_{\partial\Omega'\times (0,T)}=0 \nonumber
\end{align} 
Assume $c_i^2\in \mathcal{A}^s_0$. This equation is well posed with $u(t,x)\in C([0,T];H^s(\Omega')^3)\cap C^1([0,T];H^{s-1}(\Omega')^3))$ for $s-1>3/2$, when $||f(t,x)||_{H^1([0,T]; H^{s-1}(\Omega')^3)}$ $||u_0(x)||_{H^s((\Omega')^3)}$, $||u_1(x)||_{H^{s-1}((\Omega')^3)}$ are all bounded. The constant $T$ is finite depending on a uniform bound on the $H^1([0,T]; H^{s-1}(\Omega')^3)$ norm of $f(t,x)$, $||b_0(x)||_{H^s((\Omega')^3)}$, $||b_1(x)||_{H^{s-1}((\Omega')^3)}$, the $H^s(\Omega)^3$ norm of $c_i(x)$, $i=1,2,3$, and $m_0, m_1$. This local well posedness result does not appear to have been stated in the literature in this form and proved in the Appendix, where the dependence of the various parameters is detailed. 

We recall that as a consequence of Sobolev embedding for all $\alpha>3/2$, we have $H^{\alpha}(\Omega')\subseteq L^{\infty}(\Omega')$. We notice that because $s>5/2$, by Sobolev embedding, we automatically obtain $u(t,x)\in C([0,T];C^1(\Omega')^3)\cap C^1([0,T];C(\Omega')^3)$. For simplicity we assume $s=3$, for the rest of this article except the Appendix and while the regularity in the proof techniques for recovery of the coefficients could be reduced, it is unclear if the system data propagates regularly in any sense for $s\leq 5/2$. 

We let the vector valued source-to-solution map $\Lambda$ associated to $u$ solving \eqref{m2N} be a map which is defined by 
\begin{align*}
(\Lambda (b_0,b_1,f))=(u_1,u_2,u_3)|_{[0,T]\times\partial\Omega}
\end{align*}
The map $\Lambda$ is defined as an operator provided the input is in the regularity class in the main theorem- the trace theorem (see the Appendix, Lemma \ref{dirichlet}) gives immediately that the map is well defined with range in $L^2([0,T];L^2(\partial\Omega))$. 

Analogously we let the linear source-to-solution map $\Lambda^{lin}$ associated to $u_{lin}$ solving \eqref{m2N} with $0$ right hand side be the map of the source to trace of the solution. It is a key point that we restrict the domain of $\Lambda$ to  a subclass of data $F$ of the form $F=(b_0,b_1,f)=\epsilon F_1=\epsilon(b_0',b_1',f_1)$, with $F_1$ \emph{independent} of $\epsilon$ and such that 
\begin{align}
||b_0'||_{H^3(\Omega')^3}+||b_1'||_{H^2(\Omega')^3}+||f_1||_{L^2([0,T];H^2(\Omega')^3)}=||F_1||_*\leq 1
\end{align} and not all possible data. (The number 1 is arbitrary, it could be a different finite constant) As a consequence of the proof techniques, the domain of the operator $\Lambda^{lin}$ we determine takes a subclass of data $F$ of the form $F=F_1$ with $||F||_*\leq 1$, for a particular finite maximum $T$ as detailed below. The $T$ in consideration is then \emph{independent} of $\epsilon$. 

Let $g_0$ denote the Euclidean metric and we assume the parameter $\epsilon$ is such that $\epsilon\in (0,\epsilon_1)$, for some finite $\epsilon_1<1$. Let $T_0(\epsilon)$ be the maximal time for which the system \eqref{m2N} is well posed, which is inversely proportional to $\epsilon$. We assume $T$ fixed is such that $ T<T_0(\epsilon_1)$. (Again, the timescale $T_0$ and its dependence on $\epsilon$ is detailed in the Appendix).

Our main result is the following
\begin{thm}\label{main}
Let $\mathcal{U}_1(t,x)=(u_{11},u_{12},u_{13})$ and $\mathcal{U}_2(t,x)=(u_{21},u_{22},u_{23})$, satisfy \eqref{m2N} with distinct sound speed coefficients, $c_{i,1}$ and $c_{i,2}\in \mathcal{A}_0^3$, for $i=1,2,3$. If $\Lambda_1=\Lambda_2$ on $[0,T]\times \partial\Omega$, then $\Lambda_1^{lin}=\Lambda_2^{lin}$ on $[0,T]\times\partial\Omega$.
\end{thm}

As a result we have the following Corollary: 
\begin{cor}\label{mainC}
Assume that $\Lambda_1=\Lambda_2$ on $[0,T]\times\partial\Omega$, then $c^2_{i,1}=c^2_{i,2}$, for all $i=1,2,3$, whenever it is known that the source to solution map for the linear problem uniquely determines the conformal factors (up to a diffeomorphism).  
\end{cor}

\begin{rem}
The proof of Theorem \ref{main} does not require any assumptions on $\Omega$, only that $\overline{\Omega}$ be compact for the well-posedeness estimates in Theorem \ref{energyestimate} to hold, and that $c_i\in \mathcal{A}_0^3$. However in practice some non trapping assumptions on the domain $\Omega$ are required for the hypothesis of the Corollary \ref{mainC} to hold c.f. \cite{LO,PUV,PUV2}. These non trapping assumptions may not be required if using the boundary control method and the full source to solution map \cite{B,BK}. Typically this Corollary enforces a condition of the form $\mathrm{diam}(\Omega)\leq T$ where the diameter of $\Omega$ is taken with respect to the maximum of the sound speeds. In the Appendix we show that such a condition is possible e.g., a nonzero $\epsilon_1$ is proven to exist in the Appendix in Lemma \ref{life}.
\end{rem}

The outline of this article is as follows. Section \ref{para} gives an explict parametrix relationship to the nonlinear wave equation. Section \ref{finalsec} shows that the parametrix in powers of $\epsilon$ is in fact a good hierarchy for recovering the parameters for the nonlinear PDE and solves the problem of finding the coefficients with various data sets. This section includes examples of non-trapping metrics which satisfy the Theorem \ref{main} and Corollary \ref{mainC}. There is an Appendix on well-posedness results for the coupled system. The well posedness estimates are one of the major contributions of this article, but they are in the Appendix as the Appendix could be stand alone.  

\section{Linearsation of the Inverse Problem}\label{para}

We consider the linear homogeneous system of wave equations
\begin{align}\label{linearW}
&\partial_t^2u_i-c_i^2(x)\Delta u_i=f(t,x), \quad i=1,2,3 \quad \textrm{in}\quad (0,T)\times \Omega'\\& \nonumber
u(0,x)=b_0(x) \quad \partial_tu(0,x)=b_1(x) \quad \textrm{in} \quad \Omega' \\&
u(t,x)|_{\partial\Omega'\times (0,T)}=0 \nonumber
\end{align} 
and the linear operator $\Box_S$ which is associated to the system if we let $u=(u_1,u_2,u_3)^t$. Through abuse of notation, we let $\Box_S^{-1}F(t,x)$ denote the solution to the Cauchy problem \eqref{linearW} above. As such, $\Box_S^{-1}$ is associated to the diagonal matrix
\begin{align}
\Box_S^{-1}=\left(\begin{array}{ccc} \Box^{-1}_{c_1}& 0& 0 \\ 0 & \Box^{-1}_{c_2} & 0\\ 0 & 0 & \Box^{-1}_{c_3} \end{array}\right)
\end{align}
with $\Box^{-1}_{c_i}$, $i=1,2,3$ is the inverse operator associated to each $\Box_{c_i}=\partial_t^2-c_i^2\Delta$. For \emph{any} fixed and finite $T$ and $\beta\in \mathbb{N}$, we know from Theorem \ref{energyestimate} in the Appendix that there exists a unique $u_i=\Box^{-1}_{c_i}(b_{0i},b_{1i},f_i)$ with $u_i\in C([0,T];H^{\beta}(\Omega'))\cap C^1([0,T];H^{\beta-1}(\Omega')) $, if $f_i\in H^1([0,T];H^{\beta-1}(\Omega))$ and $\epsilon$ is sufficiently small. As a result the operator $\Box_S^{-1}$ is diagonal in each component and is a bounded operator $H^1([0,T];H^{\beta-1}(\Omega)^3)\mapsto C([0,T];H^{\beta}(\Omega')^3)\cap C^1([0,T];H^{\beta-1}(\Omega')^3)$.  

We consider the 'open source problem' for the nonlinear waves now 
\begin{align}\label{nonlinearW}
&\partial_t^2u_i-c_i^2(x)\Delta u_i=|u|^2+f(t,x),  \quad i=1,2,3\quad \textrm{in} \quad\mathbb{R}_t^+\times \Omega'\\& \nonumber
u(0,x)=b_0(x) \quad \partial_tu(0,x)=b_1(x) \quad \textrm{in} \quad \Omega' \nonumber\\&
u(t,x)|_{\partial\Omega'\times (0,T)}=0 \nonumber
\end{align} 

Let $v=(v_1,v_2,v_3)$ and $w=(w_1,w_2,w_3)$ be three component vectors and we set $N(v,w)=(v\cdot w,v\cdot w,v\cdot w)$, although this construction is applicable for any quadratic nonlinearity. 
\begin{lem}\label{parametrix}
Let $\epsilon>0$, $f_1(t,x)\in H^1([0,T];H^2(\Omega')^3)$, $b_0',b_1'$ in $H^3(\Omega')^3$, $H^2(\Omega)'$ respectively,  with $||f_1(t,x)||_{H^1([0,T];H^2(\Omega')^3)}+||b_0'||_{H^3(\Omega')^3}+||b_1'||_{H^2(\Omega')^3}=||F_1||_*\leq 1$  a parametrix solution to \eqref{nonlinearW} when $F=\epsilon F_1=\epsilon(b_0',b_1',f_1)$ with $\epsilon$ small, is represented by the following
\begin{align}\label{expansion}
w=\epsilon w_1+\epsilon^2w_2+E_{\epsilon}
\end{align}
with individual terms given by 
\begin{align}\label{pexp}
& w_1=\Box_S^{-1}F \\& \nonumber
w_2=-\Box_S^{-1}(N((0,0,w_1),(0,0,w_1))\\& \nonumber
||E_{\epsilon}||_{C([0,T];H^1(\Omega')^3)\cap C^1([0,T];L^2(\Omega')^3)}\leq 2D_1(T)^3\epsilon^3 \nonumber
\end{align}
and $w\in C([0,T];H^3(\Omega')^3)\cap C^1([0,T];H^2(\Omega')^3)$. Moreover for $F=\epsilon F_1$ we have that 
\begin{align}\label{wsm}
||w_i||_{C([0,T];H^1(\Omega')^3)\cap C^1([0,T];L^2(\Omega')^3)}\leq (D_1(T))^i \quad \forall i=1,2
\end{align}
where $D_1(T)=C_1(1+T+(1+\tilde{A}_1T)\exp(\tilde{A}_1T))\exp(\tilde{A}_1T))$ is the constant in Theorem \ref{wpN} determined by \eqref{Wwp} from Theorem \ref{energyestimate}.
\end{lem}

\begin{proof}
By plugging in \eqref{expansion} into \eqref{nonlinearW}, and matching up the terms in powers of $\epsilon$ one gets a set of recursive formulae. Solving the equations recursively gives the expansion for the coefficients. 
To prove inequality \eqref{wsm} one remarks that 
\begin{align}
||w_1||_{C([0,T];H^1(\Omega')^3)\cap C^1([0,T];L^2(\Omega')^3)}\leq D_1(||F_1||_*)
\end{align}
which is essentially inequality \eqref{Wwp} from Theorem \ref{energyestimate} in the Appendix. We use this fact and Gargliano-Nirenberg-Sobolev to see
\begin{align}
&||\Box_S^{-1}(N(w_1,w_1))||_{C([0,T];H^{1}(\Omega')^3)}\leq D_1||w_1^2||_{C([0,T];L^2(\Omega')^3)}\leq\\& 
 D_1||w_1||_{C([0,T];L^4(\Omega')^3)}^2\leq  D_1\left(||w_1||_{C([0,T];\dot{H}^1(\Omega')^3)}\right)^{3/2}\left(||w_1||_{C([0,T];L^2(\Omega')^3)}\right)^{1/2}\leq (D_1)^2
\end{align}
where in the last inequality we used the fact $x^{\alpha}$ is monotone increasing in $\alpha$ for  $\alpha\geq 0$ and the requirement $||F_1||_*\leq 1$, by our choice of domain for the operator $\Lambda$. 
 
To find a bound on the error, we see that if $u$ is the true solution to \eqref{m2N}, and $w$ is the Ansatz solution, $u-w=E_{\epsilon}(t,x)$ satisfies the equation 
\begin{align}
\Box_SE_{\epsilon}=|u|^2-|w|^2+\tilde{E}_{\epsilon}
\end{align}
where for all $i=1,2,3$
\begin{align}
\tilde{E}_{\epsilon i}=2\epsilon^3w_2\cdot w_1+\epsilon^4w_2^2
\end{align}
which implies
\begin{align}
\Box_SE_{\epsilon}=E(u+w)+\tilde{E}_{\epsilon}
\end{align}
Using \eqref{wsm}, and Theorem \ref{energyestimate}, the main part of the parametrix and error are bounded appropriately. Indeed, we have that 
\begin{align}\label{eps1}
&||E_{\epsilon}||_{C([0,T];H^1(\Omega')^3)\cap C^1([0,T];L^2(\Omega')^3)} \leq \\& \nonumber D_1(T)||E_{\epsilon}(u+w)||_{L^2([0,T];L^2(\Omega')^3)}+D_1(T)||\tilde{E}_{\epsilon}|| _{L^2([0,T];L^2(\Omega')^3)} \leq \\& \nonumber D_1(T)T||E_{\epsilon}||_{C([0,T];L^2(\Omega')^3)}||(u+w)||_{C([0,T];L^2(\Omega')^3)}+D_1(T)||\tilde{E}_{\epsilon}||_{L^2([0,T];L^2(\Omega')^3)} \leq \\& \nonumber
2T\epsilon D_1(T)||E_{\epsilon}||_{C([0,T];L^2(\Omega')^3)}+D_1(T)||\tilde{E}_{\epsilon}||_{L^2([0,T];L^2(\Omega')^3)}
\end{align}
The result follows provided 
\begin{align}
2T\epsilon D_1(T)<1
\end{align}
which is already satisfied by \eqref{cond1}. 
\end{proof}

\section{Testing of the Waves: A New Construction}\label{finalsec}

The difficulty in constructing accurate approximations to solutions of nonlinear PDE is existence of singularites which can propagate forward in time when the waves interact. When $\phi(x)$ is smooth and compactly supported, then convolution with
\begin{align}
f_k(x)=k^{d/2}\phi\left(\frac{x}{k}\right)
\end{align}
as $k\rightarrow \infty$ approximates a dirac mass $\delta_0$ with $d$ the dimension of the space in consideration. We see the function $f_k(x)$ is in $L^2(\mathbb{R}^d)$ but $f^2_k(x)$ is not when $k\rightarrow \infty$. This causes problems when considering a parametrix for a semi-linear wave equation of the form $\Box_gu=|u|^2$ and indeed, there are examples where the wave front sets of the nonlinear hyperbolic PDE do not coincide with those of the linear hyperbolic PDE, c.f. \cite{beals} Theorem 2.1 for example.  

In \cite{rr2}, they proved that the initial and subsequent crossings wave solutions to the linear PDE are the only source of nonlinear singularities. Thus, for $H^{\alpha}(\mathbb{R}^d)$ $\alpha>d/2$ compactly supported initial data we no longer have this problem, and the data propagates regularly (provided there are no derivatives in the nonlinearity). Using theorems in \cite{rr,rr2}, and \cite{beals} we could lower the assumptions on the initial data regularity for the problem, using the same techniques here, but this is not the main focus of the article. 

We show that one can recover the coefficients of the toy model for the elasticity coefficients and show that the wave interaction is nonzero given sufficient regularity. 

\begin{proof}[Proof of Theorem \ref{main}]

The components in the parametrix as in \eqref{expansion} for each of them we denote as $u_{jik}$ where $j$ denotes the vector component $j=1,2,3$, $i$ denotes the index of the system $i=1,2$ and $k$ denotes the power in the expansion of $\epsilon$ $k=1,2$. Therefore
\begin{align}
&(\mathcal{U}_1-\mathcal{U}_2)=\\& \nonumber \epsilon(u_{111}-u_{211}, u_{121}-u_{221}, u_{131}-u_{231})+ \epsilon^2(u_{112}-u_{212}, u_{122}-u_{222}, u_{132}-u_{232})+\\& \nonumber \epsilon^3(E_{\epsilon}^1-E_{\epsilon}^2)
\end{align}
where $E^0_{\epsilon}(t,x)=(E_{\epsilon}^1-E_{\epsilon}^2)$ is a three term component of the error. From Lemma \ref{parametrix}, this error is bounded by $D^3(T)\epsilon^3$ in $C([0,T];H^1(\Omega)^3)$ norm. Here is where we use the fact $u,w$ and $E_{\epsilon}$ are bounded in $C([0,T];C^1(\Omega)^3)$ norm so we know the data propagates regularly, and we do not have to check any singularity crossings at this point. 
 
If $\Lambda_1=\Lambda_2$ then it follows that $\Lambda_1^{lin}=\Lambda_2^{lin}$, by matching up the $\mathcal{O}(\epsilon)$ terms in the expansion and varying over all data $F_1$. Indeed, otherwise one has that $E_{\epsilon}, (w_{1,1}-w_{2,1})$, and $(w_{1,2}-w_{2,2})$ are all nonzero and
\begin{align}
\frac{||(w_{1,1}-w_{2,1})+\epsilon(w_{1,2}-w_{2,2})||_{L^2([0,T];L^2(\partial\Omega)^3)}}{\epsilon^2}=||E^0_{\epsilon}||_{L^2([0,T];L^2(\partial\Omega)^3)}
\end{align}
for all possible choices of data $F_1$ and for all $\epsilon$. The left hand side blows up as $\epsilon$ goes to $0$. However, the right hand side involving $E_{\epsilon}^0$ is uniformly bounded by $TD_1(T)^3<\epsilon_1^{-1}[D_1(\epsilon_1)]^{-2}\approx \epsilon_1^{-3}$ from \eqref{cond1} and Lemma \ref{dirichlet} in the Appendix. Thus this statement is impossible.  The key point is that for each $\epsilon$, the maximal lifespan of the solution is $T(\epsilon)$ with $T(\epsilon)>T(\epsilon_1)$. This is a bit tricky to understand as we restrict to $T$ such that $T<T(\epsilon_1)$, so even though a larger lifespan may exist, this is not what timescale we use for the \emph{family} of source data.  
\end{proof}

We now recall some definitions in the literature to provide an example of metrics which satisfy the necessary conditions for Theorem \ref{main}. 

\begin{defn}[Definition in \cite{UV}]
Let $(M_0,g)$ be a compact Riemannian manifold with boundary. We say that $M_0$ satisfies the foliation condition by strictly convex hyper surfaces if $M_0$ is equipped with a smooth function $\rho: M_0\rightarrow [0,\infty)$ which level sets $\sigma_t=\rho^{-1}(t),\,\, t<T$ with some $T>0$ finite, are strictly convex as viewed from $\rho^{-1}((0,t))$ for $g$ , $d\rho$ is non-zero on these level sets, and $\Sigma_0=\partial M_0$ and $M_0\setminus \bigcup_{t\in [0,T)}\Sigma_t$ has empty interior. 
\end{defn}

The global geometric condition of \cite{UV} is a natural analog of the condition
\begin{align}\label{decay}
\frac{\partial}{\partial r}\frac{r}{c(r)}>0
\end{align}
with 
\begin{align}
\frac{\partial}{\partial r}=\frac{x}{|x|}\cdot \partial_x 
\end{align}
the radial derivative as proposed by Herglotz \cite{Herglotz} and Wiechert $\&$ Zoeppritz \cite{WZ} for an isotropic radial sound speed $c(r)$. In this case the geodesic spheres are strictly convex. 

In fact \cite{PUV2}, cite Sec 6. extends the Herglotz and Wiechert $\&$ Zoeppritz results to not necessarily radial speeds $c(x)$ which satisfy the radial decay condition \eqref{decay}. Let $B(0,R)$ $R>0$ be the ball in $\mathbb{R}^d$ with $d\geq 3$ which is entered at the origin with radius $R>0$. Let $0<c(x)$ be a smooth function in $B(0,R)$. 
\begin{prop}
The Herglotz and Wieckert $\&$ Zoeppritz condition is equivalent to the condition that the Euclidean spheres $S_r=\{|x|=r\}$ are strictly convex in the metric $c^{-2}\,dx^2$ for $0<r\leq R$. 
\end{prop}

\begin{exmp} [Herglotz Wiechert and Zoeppritz Systems] 
Let $\Omega$ be the unit ball, so $M_0=\overline{\Omega}$ then for any $c_i\in C^3(\Omega)$, $i=1,2,3$ such that  
\begin{align}
\frac{1}{1+r^2}\leq c_i(r)\leq 1
\end{align}
satisfy the convexity condition \eqref{decay}, and the conditions of Theorem \ref{main} for equations of the form \eqref{m2N}. Using known results on injectivity in \cite{PUV2}, systems with coefficients of this type provide an example of a case where Corollary \ref{mainC} holds. Here we remark that $\partial_t^2-c^2\Delta$ and $\partial_t^2-\Delta_g$ have the same principal symbols if $g=c^{-2}dx$ and $c^2\in \mathcal{A}_0^3$ (they coincide in dimension 2). In particular, in \cite{PUV2} they show for the scalar valued wave equation with $f_1(t,x)=0$, 
\begin{align}\label{system}
&\partial_t^2u-c^2(x)\Delta_x u=0 \nonumber \quad \textrm{in}\,\, [0,T]\times \Omega', \\& \nonumber
u(0,x)=u_0(x) \quad \partial_tu(0,x)=u_1(x) \quad \textrm{in}\quad (\Omega')\\&
u(t,x)|_{\partial\Omega'\times[0,T]}=0
\end{align}
in $\mathbb{R}^3$, that the linear source to solution map $\Lambda$ is enough to determine the lens relation on the subset $\Omega$. For sound speeds of the above form, they can reconstruct the the sound speed from the lens relation. 
\end{exmp}

\section{Appendix: Well-posedness estimates for the semi-linear wave equations}

We set $\overline{\Omega}\subset \Omega'$, where $\Omega'$ is a larger domain in $\mathbb{R}^3$, with Dirichlet boundary conditions. In the appendix, we prove the following theorem:
\begin{thm}\label{wpN}
Let $s>5/2$ be an arbitrary integer. Assume that $c_i(x)\in \mathcal{A}_0^s, \forall i=1,2,3$. Let $F(t,x)=(u_0,u_1,f)=\epsilon F_1(t,x)=\epsilon(b_0,b_1,f_1)$ with $||b_0||_{H^s(\Omega')}+||b_1||_{H^{s-1}(\Omega')}+||f_1||_{L^2([0,T];H^{s-1}(\Omega'))}=||F_1(t,x)||_*\leq 1$, then there exists a unique solution $u(t,x)$ with $u(t,x)\in C([0,T]; H^s(\Omega')^3)\cap C^1([0,T]; H^{s-1}(\Omega')^3)$ to the coupled system:
\begin{align}\label{system}
&\partial_t^2u_i-c_i^2(x)\Delta_x u_i=|u|^2+f(t,x) \nonumber \quad \textrm{in}\,\, [0,T]\times \Omega', \quad i=1,2,3 \\& \nonumber
u(0,x)=u_0(x) \quad \partial_tu(0,x)=u_1(x) \quad \textrm{in}\quad (\Omega')^3\\&
u(t,x)|_{\partial\Omega'\times[0,T]}=0
\end{align}
provided $C(s)T<\log((3\epsilon)^{-1})-C'(s)$ where $C(s), C'(s)$ depend on $s$ and the $C^s(\Omega')$ norm of the $c_i's$.
\end{thm}  

We prove the local well posedness theorem via an abstract Duhamel iteration argument. We recall Duhamel's principle.
\begin{defn}[Duhamel's principle]
Let $\mathcal{D}$ be a finite dimensional vector space, and let $I$ be a time interval. The point $t_0$ is a time $t$ in $I$. The operator $L$ and the functions $v,f$ are such that:
\begin{align}
L\in \mathrm{End}(\mathcal{D})
\quad v\in C^1(I\rightarrow \mathcal{D}), \quad f\in C^0(I\rightarrow \mathcal{D})
\end{align}
then we have that 
\begin{align}
\partial_tv(t)-Lv(t)=f(t) \quad \forall t\in I
\end{align}
if and only if 
\begin{align}
v(t)=\exp(t-t_0)Lv_0(t_0)+\int\limits_{t_0}^t\exp((t-s)L)f(s)\,ds \quad \forall t\in I
\end{align}
\end{defn}

We view the general equation as 
\begin{align}
v=v_{lin}+JN(f).
\end{align}
with $J$ a linear operator. 
We also have the following abstract iteration result: 
\begin{lem}\label{abstract}[\cite{tao} Prop 1.38]
Let $\mathcal{N},\mathcal{S}$ be two Banach spaces and suppose we are given a linear operator $J:\mathcal{N}\rightarrow\mathcal{S}$ with the bound
\begin{align}
||JF||_{\mathcal{S}}\leq C_0||F||_{\mathcal{N}}
\end{align}
for all $F\in\mathcal{N}$ and some $C_0>0$. Suppose that we are given a nonlinear operator $N:\mathcal{S}\rightarrow\mathcal{N}$ which is a sum of a $u$ dependent part and a $u$ independent part. Assume the $u$ dependent part $N_u$ is such that $N_u(0)=0$ and obeys the following Lipschitz bounds
\begin{align}
||N(u)-N(v)||_{\mathcal{N}}\leq \frac{1}{2C_0}||u-v||_{\mathcal{S}}
\end{align}
for all $u,v\in B_{\epsilon}=\{u\in\mathcal{S}: ||u||_{S}\leq\epsilon\}$ for some $\epsilon>0$. In other words we have that $||N||_{\dot{C}^{0,1}(B_{\epsilon}\rightarrow\mathcal{N})}\leq \frac{1}{2C_0}$. then, for all $u_{lin}\in B_{\epsilon/2}$ there exists a unique solution $u\in B_{\epsilon}$ with the map $u_{lin}\mapsto u$ Lipschitz with constant at most $2$. In particular we have that
\begin{align}
||u||_{\mathcal{S}}\leq 2||u_{lin}||_{\mathcal{S}}. 
\end{align}
\end{lem}
 
We start by proving general energy estimates for the linear problem. We have the following classical result, for all $\beta\in \mathbb{N}$. 
\begin{thm}\label{energyestimate}
Let $c\in\mathcal{A}_0^{\beta}$, and $f(t,x)\in L^2([0,T];H^{\beta-1}(\Omega'))$, $||u_0(x)||\in H^s((\Omega')^3)$, $||u_1(x)||\in H^{s-1}((\Omega')^3)$. If $u$ is a solution to 
\begin{align}\label{wavec}
&\partial_t^2u-c^2(x)\Delta u=f(t,x) \quad \mathrm{in}\quad [0,T]\times \Omega' \\& \nonumber
\partial_tu(0,x)=u_1(x) \quad u(0,x)=u_0(x) \,\,\mathrm{in}\,\, \Omega' \nonumber \\&
u(t,x)=0 \quad \mathrm{on} \quad [0,T]\times \partial\Omega' \nonumber
\end{align}
we have the following set of estimates:
\begin{itemize}
\item There exists $C$ depending on $m_0$ and $||c^2||_{C^1(\Omega')}$ and $\tilde{A}_1$ depending on $||c^2||_{C^1(\Omega')}$ such that
\begin{align}
&||u||_{C([0,T];\dot{H}^1(\Omega'))\cap C^1([0,T];L^2(\Omega'))}\leq \\& \nonumber C\left(||u_0||_{H^1(\Omega')}+||u_1||_{L^2(\Omega')}+||f(t,x)||_{L^2(\Omega' \times [0,T])}\right)\exp(\tilde{A}_1T). 
\end{align}
and
\item There exists $C_1$ which depends on $m_0$ and $||c_i^{2}(x)||_{H^{\beta}(\Omega')}$ and $\tilde{A}_{\beta}$ which depends on $||c_i^{2}(x)||_{H^{\beta}(\Omega')}$ such that
\begin{align}\label{Wwp}
&||u||_{C([0,T]; H^{\beta}(\Omega'))}+||\partial_tu||_{C([0,T]; H^{\beta-1}(\Omega'))}\leq \\& \nonumber
C_1(1+T)\exp(\tilde{A}_{\beta}T)\times \\& \nonumber
(||u_0||_{H^{\beta}(\Omega')}+||u_1||_{H^{\beta-1}(\Omega')}+\tilde{A}_{\beta}T(||u||_{C([0,T]; H^{\beta-1}(\Omega'))}+ ||\partial_tu||_{C([0,T]; H^{\beta-2}(\Omega'))})+\\& \nonumber ||f||_{L^2([0,T]; H^{\beta-1}(\Omega'))})
\end{align}
\end{itemize}
\end{thm}
\begin{proof}
The proofs below are loosely based on Theorem 4.6 and Corollary 4.9 in \cite{Luk} which have been adapted for our setting. By definition we have
\begin{align}\label{defwave}
\int\limits_0^t\int\limits_{\Omega'}(\partial_s^2u-c^2\Delta u)\partial_su\,dx\,ds=\int\limits_0^t\int\limits_{\Omega'}f(s,x)\partial_su\,dx\,ds
\end{align}
and
\begin{align}
\nabla\cdot(c^2\nabla u)=c^2\Delta u+\nabla c^2\cdot\nabla u.
\end{align}

We also have by the divergence theorem 
\begin{align}
&\int\limits_0^t\int\limits_{\Omega'}\partial_su(\nabla\cdot(c^2\nabla u))\,dx\,ds=\\& \nonumber
-\int\limits_0^t\int\limits_{\Omega'}\partial_s(\nabla u)\cdot (c^2\nabla u)\,dx\,ds+\int\limits_0^t\int\limits_{\partial \Omega'}\partial_s u\frac{\partial(c^2u)}{\partial \nu}\,dS\,ds.
\end{align}

We set 
\begin{align}
||u||_E^2(t)=\frac{1}{2}\left(\int\limits_0^t\int\limits_{\Omega'}|\nabla u(s,x)|^2+|\partial_su(s,x)|^2\,dx\,ds\right) 
\end{align}
and 
\begin{align}
||u||_{Ec}^2(t)=\frac{1}{2}\left(\int\limits_0^t\int\limits_{\Omega'}c^2|\nabla u(s,x)|^2+|\partial_su(s,x)|^2\,dx\,ds\right) 
\end{align}

The end result of plugging the equalities into \eqref{defwave} is that 
\begin{align}\label{inept}
\frac{d}{ds}||u||_{Ec}^2(T)=\int\limits_0^T\int\limits_{\Omega'}f\partial_su\,dx\,ds+\int\limits_0^T\int\limits_{\partial \Omega'}\partial_s u\frac{\partial(c^2u)}{\partial \nu}\,dS\,ds-\int\limits_0^T\int\limits_{\Omega'}\nabla c^2\cdot\nabla u\partial_su\,dx\,ds
\end{align} 
We let $C=\min\{m_0,1\}$. Taking the absolute values of both sides and remarking that $2ab\leq a^2+b^2$ for all real valued functions $a,b$ we obtain 
\begin{align}
C\frac{d}{dt}||u||^2_E(T)\leq \tilde{A}||f||^2_{L^2(\Omega'\times[0,T])}+\tilde{A}||u||^2_E(T)
\end{align}
 Applying Grownwall's inequality gives the desired result. 
For the second estimate, differentiating the equation \eqref{defwave} (e.g. applying the operator $\nabla^k$ successively) gives control over
\begin{align}
||u||_{C([0,T];\dot{H}^k(\Omega'))\cap C^1([0,T];\dot{H}^{k-1}(\Omega'))} 
\end{align}
it remains to control $||u||_{C([0,T];L^2(\Omega'))}$ but it is easy to see as 
\begin{align}
||u||_{C([0,T];L^2(\Omega'))}\leq ||u_0||_{L^2(M)}+\int\limits_0^T||\partial_tu||_{L^2(\Omega')}^2(t)\,dt
\end{align}
which gives the desired result. 
\end{proof}

\begin{proof}[Proof of Theorem \ref{wpN}]
Recall that $H^{\alpha}(M)\subseteq L^{\infty}(M)$ if $\alpha>d/2$, which is an assumption we will use here. 
If we reformulate the wave equation \eqref{wavec} as
\begin{align}
\left(\begin{array}{c} u \\ v \end{array}\right)_t  = 
\left(\begin{array}{cc} 0 & 1\\ c^2\Delta & 0 \end{array}\right)
\left(\begin{array}{c} u \\ v \end{array}\right)+\left(\begin{array}{c} 0 \\ f \end{array}\right)  
\end{align}
with 
\begin{align}\label{variablesystem1}
&\mathcal{U}=\left(\begin{array}{c} u \\ v \end{array}\right)\quad
A=\left(\begin{array}{cc} 0 & 1\\ c^2\Delta & 0 \end{array}\right) \quad F=\left(\begin{array}{c} 0 \\ f \end{array}\right)  \quad \Phi=\left(\begin{array}{c} g \\ h \end{array}\right)
\end{align}
One can write the inhomogenous scalar valued wave equation as
\begin{align}
&\mathcal{U}_t=A\mathcal{U}+F\\& \nonumber
\mathcal{U}(0)=\Phi
\end{align}

Using this as our model, can re-write the more complicated system \eqref{system} 
\begin{align}
&\mathcal{W}_t=\tilde{A}\mathcal{W}+\tilde{F}\\& \nonumber
\mathcal{W}(0)=0
\end{align}
with 
\begin{align}
\mathcal{W}=(u_1,v_1,u_2,v_2,u_3,v_3)^t \quad
\tilde{F}=(0,|u|^2,0,|u|^2,0,|u|^2)^t+(0,\epsilon f_{1i},0,\epsilon f_{1i},0,\epsilon f_{1i})
\end{align}
and
\begin{align}
A_i=\left(\begin{array}{cc} 0 & 1\\ c_i^2\Delta & 0 \end{array}\right)
\end{align}
elements of the block diagonal matrix
\begin{align}
\tilde{A}=\left(\begin{array}{ccc} A_1 & 0& 0 \\ 0 & A_2 & 0\\ 0 & 0 & A_3
 \end{array}\right)
\end{align}
We then apply the abstract Duhamel iteration argument with $\mathcal{S}=C([0,T];H^{s}(\Omega')^6)\cap C^1([0,T];H^{s-1}(\Omega')^6)$ and $\mathcal{N}=L^2([0,T];H^{s-1}(\Omega')^6)$. We leave the $s$ as an arbitrary integer, so if we set $J$ the Duhamel propagator associated to $\Box_S$ with $0$ initial conditions, then the inequality $||JF||_{\mathcal{S}}\leq C_0||F||_{\mathcal{N}}$ is easily satisfied with $C_0=D_s(T)$ given to us by Theorem \ref{energyestimate} (The constant $D_s(T)$ is the maximum over the conformal factors). In practice for the rest of the article we only need $s=3$. 

The key observation is that 
\begin{align}\label{diff}
||\tilde{F}(W_1)-\tilde{F}(W_2)||_{\mathcal{N}}\leq B||W_1-W_2||_{\mathcal{S}}.
\end{align}
for some positive constant $B$, depending on $||c_i^2(x)||_{H^s(\Omega)}$ $i=1,2,3$, $m_0$ and $m_1$, $\epsilon$ and $T$ with $W_1=(w_{1,1},v_{1,1},w_{1,2},v_{1,2}, w_{1,3},v_{1,3})^t$ and $W_2=(w_{2,1},v_{3,1},w_{2,2},v_{2,2}, w_{2,3},v_{2,3})^t$.

By linearity, we have  
\begin{align*}
& ||\tilde{F}(W_1)-\tilde{F}(W_2)||_{\mathcal{N}}\leq \\& \nonumber 3||w_{1,1}^2-w_{2,1}^2+w_{1,2}^2-w_{2,2}^2+w_{1,3}^2-w_{2,3}^2||_{L^2([0,T]; H^{s-1}(\Omega'))}. 
\end{align*}
We set
\begin{align}
M_1=\sup_{i=1,2; 1\leq j\leq 3}||w_{i,j}||_{C([0,T]; H^{s-1}(\Omega'))}\leq \epsilon 
\end{align}
where we used the upper bound implied by the hypothesis $\mathcal{W}_1,\mathcal{W}_2\in B_{\epsilon}$. We then obtain
\begin{align}
&||\tilde{F}(W_1)-\tilde{F}(W_2)||_{\mathcal{N}}\leq 3\epsilon T\sum\limits_{i=1}^3||w_{1,i}-w_{2,i}||_{C([0,T];H^{s-1}(\Omega'))}\leq \\& \nonumber \frac{3}{2}\epsilon T||W_1-W_2||_{\mathcal{S}}
\end{align}
and the result \eqref{diff} follows with $B=\frac{3}{2}\epsilon T$. 

The corresponding Duhamel iterates are
\begin{align}
\mathcal{W}^0=\mathcal{W}_{lin} \quad
\mathcal{W}^n=\mathcal{W}^{n-1}_{lin}+JN(\mathcal{W}^{n-1})
\end{align}
and from Lemma \ref{abstract} we can conclude
\begin{align}
\lim\limits_{n\rightarrow\infty}\mathcal{W}^n=\mathcal{W}^*
\end{align}
 is the unique solution $W^*\in B_{\epsilon}$ whenever $T$ is sufficiently small, by Lemma \ref{abstract}. In particular, for the Theorem to hold we must have
\begin{align}
\frac{3T\epsilon}{2}<\frac{1}{2D_s(T)}.
\end{align} 
\begin{align}\label{cond1}
TD_s(T)<(3\epsilon)^{-1}
\end{align}
As $D_s(T)$ is a polynomial in $T$ and $\exp(\tilde{A}T)$ and since $\log(R)\leq R$ for all $R\in\mathbb{R}^+$, 
\begin{align}\label{cond2}
C(s)T<\log((3\epsilon)^{-1})-C(s')
\end{align} 
for some $C(s),C(s')$ depending on $s$ and $\tilde{A}_s$. 
\end{proof}

\begin{lem}\label{life}
Let $T(\epsilon)$ denote the maximal timespan for well-posedness of the system \eqref{m2N}. There exists $\epsilon_1\in (0,1)$ such that for all $\epsilon\in(0,\epsilon_1)$, the inequality
\begin{align}\label{lifespan}
\mathrm{diam}(\Omega)<T(\epsilon_1)<T(\epsilon)
\end{align} 
holds. 
\end{lem}
\begin{proof}
For each $\epsilon$, we know the timescale $T(\epsilon)$ must be such that \eqref{cond2} holds with $s=3$. Then the condition \eqref{lifespan} is satisfied if \eqref{cond2} holds with $T$ replaced by $\mathrm{diam}(\Omega)$. 
This is clearly possible as $\mathrm{diam}(\Omega)$ is finite, whence the conclusion is possible. 
\end{proof}

\begin{lem}\label{dirichlet}
The operator $\Lambda$ is bounded on $u\in L^2([0,T];H^1(\Omega))\cap C([0,T];C(\overline{\Omega}))$, with 
\begin{align}
||\Lambda u||_{L^2([0,T];L^2(\partial\Omega))}\leq C||u||_{L^2([0,T];H^1(\Omega))}
\end{align}  
where $C$ is a constant depending only on the geometry of $\Omega$. 
\end{lem}

We recall the trace theorem 
\begin{thm}
Assume that $\Omega$ is a bounded with Lipschitz boundary, then $\exists$ a bounded linear operator
\begin{align}
Tv=v|_{\partial\Omega} \quad \textrm{for} \quad v\in W^{1,p}(\Omega)\cap C(\overline{\Omega})
\end{align}
and a constant $c(p,\Omega)$ depending only on $p$ and the geometry of $\Omega$ such that
\begin{align}
||Tv||_{L^p(\partial\Omega)}\leq c(p,\Omega)||v||_{W^{1,p}(\Omega)}
\end{align} 
\end{thm}
The proof of Lemma \ref{dirichlet} now follows immediately.

\section*{Acknowledgments}
A.~W.~acknowledges support by EPSRC grant EP/L01937X/1.

\end{document}